\documentclass[11pt]{amsart}

\usepackage{amsfonts,amsthm,amsmath,enumerate,amssymb,latexsym,color,tcolorbox,tikz,mathrsfs,bm}
\usetikzlibrary{bending}
\usepackage[text={6in,9in},centering]{geometry}
\usepackage[colorlinks,
            linkcolor=blue,
            anchorcolor=red,
            citecolor=red]{hyperref}

\graphicspath{{figures/}}

\theoremstyle{plain}
\newtheorem{theorem}{Theorem}
\newtheorem{lemma}[theorem]{Lemma}

\newtheorem{conjecture}[theorem]{Conjecture}

\theoremstyle{definition}

\theoremstyle{remark}

\numberwithin{equation}{section}

\def\R{\mathbb R}

\begin{document}

\title[A self-similar set with non-locally connected components]{A self-similar set with non-locally \\ connected components}

\author{Jian-Ci Xiao}
\address{Department of Mathematics, The Chinese University of Hong Kong, Shatin, Hong Kong}
\email{jcxiao1994@gmail.com}



\begin{abstract}
    Luo, Rao and Xiong [Topol. Appl. {\bf 322} (2022), 108271] conjectured that if a planar self-similar iterated function system with the open set condition does not involve rotations or reflections, then every connected component of the attractor is locally connected. We create a homogeneous counterexample of Lalley-Gatzouras type, which disproves this conjecture.
\end{abstract}

\maketitle

\section{Introduction}

There have been a number of advances in the field of fractal geometry that improve our understanding of the fractal dimensions of fractal sets and measures, but much less is known about the topology. Even some fundamental topological properties such as the connectedness and local connectedness have not been well studied. Hata~\cite{Hat85} related the connectedness of a self-similar set to the connectedness of some graph structure of the corresponding iterated function system (or simply IFS). When the self-similar set enjoys nice separation conditions, the graph structure can be thoroughly identified. In fact, this often provides not only an approach for the verification of the connectedness (\cite{DLRWX21,DL11,RX22}), but also useful information about the connected components of the original set (\cite{LLR13,LX22,XX09,Xiao21}).

The local connectedness of fractal sets and their subsets is far more intricate and poorly understood. Recall that a space $X$ is said to be locally connected at $x\in X$ if for every neighborhood $U$ of $x$ , there exists a connected neighborhood $V$ of $x$ contained in $U$. If $X$ is locally connected at each of its points, it is said to be \emph{locally connected}. A pioneering result was given by Hata~\cite{Hat85} stating that if a self-similar set is connected, then it must be locally and path connected. In~\cite{LRX22}, Luo, Rao and Xiong further demonstrated that a self-similar set is locally connected if and only if it has only finitely many connected components. There also exists some work~\cite{Luo07,LRT02,NT04,Tan05,TZ24} concerning the local connectedness of self-similar or self-affine tiles and related sets, e.g., their boundaries, the closure of their disjoint interiors or their complements.

In addition to the local connectedness of a given self-similar set, it is also interesting to ask whether its connected components are locally connected. Luo, Rao and Xiong~\cite{LRX22} showed that every connected component of any fractal square (or generalized Sierpi\'nski carpet) is locally connected. They also provided an example where this property does not hold when the IFSs involve rotations. The idea can be summarized as follows: consider a carpet-like self-similar set that contains infinitely many vertical line segments, including the left and right sides of the unit square, and require these segments to belong to different components. Then add a pair of up-down adjacent copies to the original system, where the top copy is exactly a $90^\circ$ rotation of the bottom one. As a result, we have a horizontal line segment gluing together infinitely many vertical segments, and the component containing them cannot be locally connected.


Since all the prior examples require the corresponding IFS to either involve rotations or to violate the open set condition (for the definition, please see~\cite{Fal14}), Luo, Rao and Xiong proposed the following conjecture.

\begin{conjecture}[\cite{LRX22}]
    Let $\Phi$ be a self-similar IFS on $\R^2$ and $K$ its attractor. If $\Phi$ satisfies the open set condition and if every similitude in $\Phi$ involves neither rotation nor reﬂection, then every component of $K$ is locally connected.
\end{conjecture}

We are able to construct a homogeneous self-similar IFS on $\R^2$ satisfying the open set condition (even the convex open set condition, that is, the invariant open set is a convex set) without using any rotations or reflections, such that some particular connected component of the attractor is not locally connected. As a result, only very specific self-similar sets that are generated from a strict grid structure or totally disconnected can have the nice property that every connected component is locally connected.

\section{The counterexample}

Consider the planar self-similar IFS $\Phi=\{\varphi_i(\mathbf{x})=\frac{1}{6}\mathbf{x}+a_i\}_{i=1}^{24}$, where 
\begin{equation*}
    \begin{gathered}
        a_1=(0,0), a_2=(0,\tfrac{1}{6}), a_3=(0,\tfrac{1}{3}), a_4=(\tfrac{1}{24},\tfrac{1}{2}), a_5=(\tfrac{1}{12},\tfrac{2}{3}), a_6=(\tfrac{1}{8},\tfrac{5}{6}), \\
        a_7=(\tfrac{1}{6},0), a_8=(\tfrac{5}{24},\tfrac{1}{6}), a_9=(\tfrac{1}{4},\tfrac{1}{3}), a_{10}=(\tfrac{7}{24},\tfrac{1}{2}), a_{11}=(\tfrac{1}{3},\tfrac{2}{3}), a_{12}=(\tfrac{3}{8},\tfrac{5}{6}), \\
        a_{13}=(\tfrac{5}{12},0), a_{14}=(\tfrac{11}{24},\tfrac{1}{6}), a_{15}=(\tfrac{1}{2},\tfrac{1}{3}), a_{16}=(\tfrac{13}{24},\tfrac{1}{2}), a_{17}=(\tfrac{7}{12},\tfrac{2}{3}), a_{18}=(\tfrac{5}{8},\tfrac{5}{6}), \\
        a_{19}=(\tfrac{5}{6},0), a_{20}=(\tfrac{5}{6},\tfrac{1}{6}), a_{21}=(\tfrac{5}{6},\tfrac{1}{3}), a_{22}=(\tfrac{5}{6},\tfrac{1}{2}), a_{23}=(\tfrac{5}{6},\tfrac{2}{3}), a_{24}=(\tfrac{5}{6},\tfrac{5}{6}).
    \end{gathered}
\end{equation*}
Please see Figure~\ref{fig:exa} for an illustration. It is easy to check that $\bigcup_{i=1}^{24}\varphi_i([0,1]^2)\subset[0,1]^2$ and $\varphi_1([0,1]^2),\ldots,\varphi_{24}([0,1]^2)$ have disjoint interiors. In particular, $\Phi$ satisfies the (convex) open set condition. Denote by $K$ the attractor associated with $\Phi$.  Such a self-similar set is sometimes refered to as a \emph{Lalley-Gatzouras carpet} (\cite{LG92}).

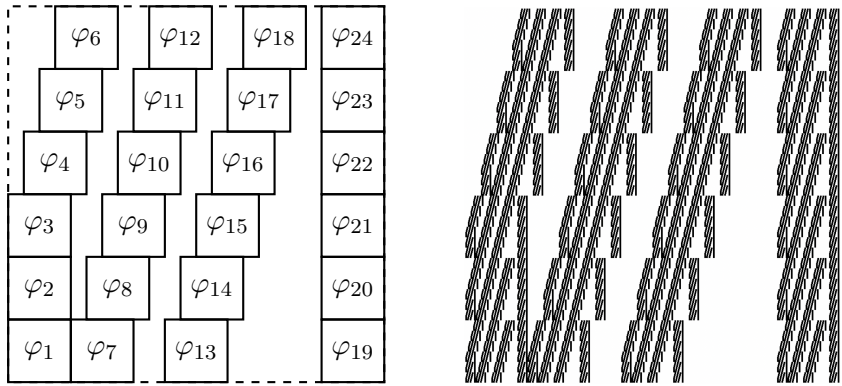
\begin{figure}[htbp]
    \centering
    \begin{tikzpicture}[scale=5]
        \draw[thick,dashed] (0,0) rectangle (1,1);
        \draw[thick] (0,0) rectangle (1/6,1/6);
        \draw[thick] (0,1/6) rectangle (1/6,1/3);
        \draw[thick] (0,1/3) rectangle (1/6,1/2);
        \draw[thick] (1/24,1/2) rectangle (5/24,2/3);
        \draw[thick] (1/12,2/3) rectangle (1/4,5/6);
        \draw[thick] (1/8,5/6) rectangle (7/24,1);
        \draw[thick] (1/6,0) rectangle (1/3,1/6);
        \draw[thick] (5/24,1/6) rectangle (3/8,1/3);
        \draw[thick] (1/4,1/3) rectangle (5/12,1/2);
        \draw[thick] (7/24,1/2) rectangle (11/24,2/3);
        \draw[thick] (1/3,2/3) rectangle (1/2,5/6);
        \draw[thick] (3/8,5/6) rectangle (13/24,1);
        \draw[thick] (5/12,0) rectangle (7/12,1/6);
        \draw[thick] (11/24,1/6) rectangle (5/8,1/3);
        \draw[thick] (1/2,1/3) rectangle (2/3,1/2);
        \draw[thick] (13/24,1/2) rectangle (17/24,2/3);
        \draw[thick] (7/12,2/3) rectangle (3/4,5/6);
        \draw[thick] (5/8,5/6) rectangle (19/24,1);
        \draw[thick] (5/6,0) rectangle (1,1/6);
        \draw[thick] (5/6,1/6) rectangle (1,1/3);
        \draw[thick] (5/6,1/3) rectangle (1,1/2);
        \draw[thick] (5/6,1/2) rectangle (1,2/3);
        \draw[thick] (5/6,2/3) rectangle (1,5/6);
        \draw[thick] (5/6,5/6) rectangle (1,1);

        \node at(1/12,1/12) {$\varphi_1$};
        \node at(1/12,1/4) {$\varphi_2$};
        \node at(1/12,5/12) {$\varphi_3$};
        \node at(1/8,7/12) {$\varphi_4$};
        \node at(1/6,3/4) {$\varphi_5$};
        \node at(5/24,11/12) {$\varphi_6$};
        \node at(1/4,1/12) {$\varphi_7$};
        \node at(7/24,1/4) {$\varphi_8$};
        \node at(1/3,5/12) {$\varphi_9$};
        \node at(3/8,7/12) {$\varphi_{10}$};
        \node at(5/12,3/4) {$\varphi_{11}$};
        \node at(11/24,11/12) {$\varphi_{12}$};
        \node at(1/2,1/12) {$\varphi_{13}$};
        \node at(13/24,1/4) {$\varphi_{14}$};
        \node at(7/12,5/12) {$\varphi_{15}$};
        \node at(5/8,7/12) {$\varphi_{16}$};
        \node at(2/3,3/4) {$\varphi_{17}$};
        \node at(17/24,11/12) {$\varphi_{18}$};
        \node at(11/12,1/12) {$\varphi_{19}$};
        \node at(11/12,1/4) {$\varphi_{20}$};
        \node at(11/12,5/12) {$\varphi_{21}$};
        \node at(11/12,7/12) {$\varphi_{22}$};
        \node at(11/12,3/4) {$\varphi_{23}$};
        \node at(11/12,11/12) {$\varphi_{24}$};
    \end{tikzpicture}
    \hspace*{0.8cm}
    \includegraphics[width=5cm]{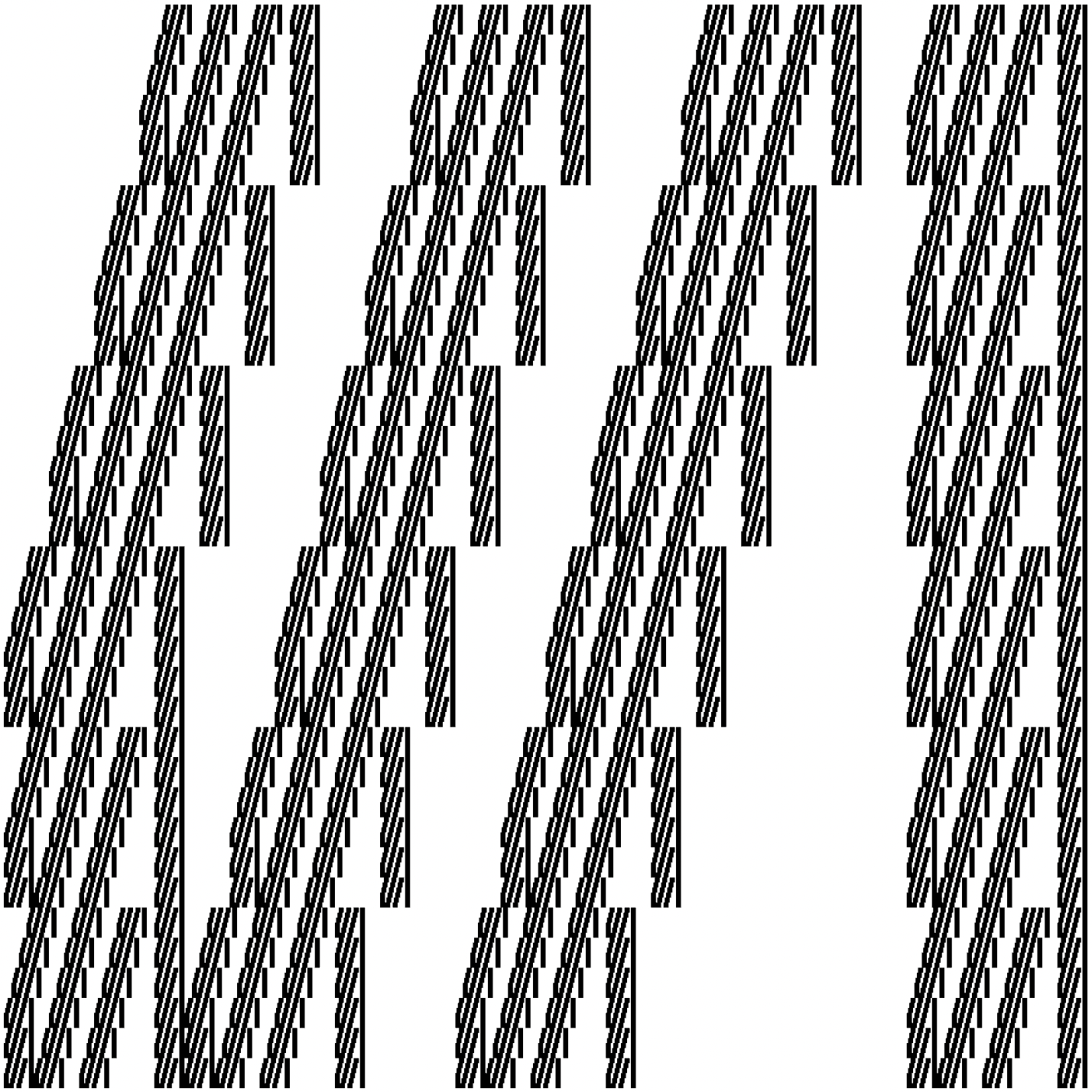}
    \caption{An illustration of the IFS $\Phi$ and the attractor $K$}
    \label{fig:exa}
\end{figure}

We will first prove that there are an infinite number of long, `well separated' line segments all hitting the left side of the unit square, which is the key ingredient.

Let $C_1$ be the Cantor set generated by the IFS $\{\frac{1}{6}x+\frac{1}{6},\frac{1}{6}x+\frac{5}{12}\}$. Recalling the translation vectors $a_7,a_{13}$, we have $C_1\times\{0\}\subset K$. For $x\in C_1$, write 
\[
    \ell_{x} = \{(x+\tfrac{y}{4},y): y\in[0,1]\}
\]
to be the line segment passing through $(x,0)$ and of slope $4$. Note that the other endpoint of $\ell_x$ is $(x+\frac{1}{4},1)$. 

\begin{lemma}
    $\bigcup_{x\in C_1} \ell_x \subset K$.
\end{lemma}
\begin{proof}
    Fix any $x\in C_1$ and $y\in [0,1]$. Write $x=\sum_{t=1}^\infty x_t6^{-t}$ where $x_t\in\{1,\frac{5}{2}\}$, and write $y=\sum_{t=1}^\infty y_t6^{-t}$ where $y_t\in\{0,1,\ldots,5\}$. A direct calculation gives us 
    \[
        \varphi_{6\lfloor x_t\rfloor+y_t+1}(0,0)=\Big( \frac{x_t}{6}+\frac{y_t}{24}, \frac{y_t}{6} \Big), \quad t\geq 1.
    \]
    Therefore, 
    \begin{align*}
        \lim_{n\to\infty} \varphi_{6\lfloor x_1\rfloor+y_1+1}\circ\cdots\circ\varphi_{6\lfloor x_n\rfloor+y_n+1}(0,0) &= \lim_{n\to\infty} \sum_{t=1}^n 6^{-t+1}\cdot\Big( \frac{x_t}{6}+\frac{y_t}{24}, \frac{y_t}{6} \Big) \\
        &= \Big( \sum_{t=1}^\infty x_t6^{-t}+\frac{1}{4}\sum_{t=1}^\infty y_t6^{-t}, \sum_{t=1}^\infty y_t6^{-t} \Big) \\
        &= (x+\frac{y}{4},y).
    \end{align*}
    That is to say, $(x+\frac{y}{4},y)\in K$. Since $x,y$ are arbitrary, $\bigcup_{x\in C_1}\ell_x\subset K$.
\end{proof}

From the construction and the illustration in Figure~\ref{fig:exa}, it is not hard to see that there is a vacant strip (avoiding $K$) between each pair of different segments in $\{\ell_x:x\in C_1\}$. We will omit the tedious computation here. In particular, each pair of these segments belong to different connected components of $K$.


Next, let us look at the left side of $[0,1]^2$. Let $C_2$ be the Cantor set generated by the IFS $\{\frac{1}{6}x,\frac{1}{6}x+\frac{1}{6}\}$. Recalling the translation vectors $a_1,a_2$, we have $\{0\}\times C_2\subset K$. For $z=\sum_{t=1}^\infty z_t6^{-t}\in C_2$ (where $z_t\in\{0,1\}$), it is convenient to use the following notations.
\begin{itemize}
    \item $x(z)=\sum_{t=1}^\infty (-\frac{3}{2}z_t+\frac{5}{2})6^{-t}$. Since $-\frac{3}{2}z_t+\frac{5}{2}=\frac{5}{2}$ when $z_t=0$ and $-\frac{3}{2}z_t+\frac{5}{2}=1$ when $z_t=1$, the map $z\mapsto x(z)$ is a bijection between $C_2$ and $C_1$. Roughly speaking, $x(z)$ can be thought of as the element in $C_1$ that is, in a certain sense, the `inverse' of $z$.
    \item $E_z=\bigcup_{i=3}^6\varphi_i(\ell_{x(z)})$.
    \item $\sigma(z)=\sum_{t=1}^\infty z_{t+1}6^{-t}$ (so $\sigma$ behaves like the usual left shift on symbolic spaces).
    \item $L_0(z)=\varphi_2(\ell_{\frac{1}{6}x(\sigma(z))+\frac{1}{6}})\cup\varphi_1(E_{\sigma(z)})$. Recall that $\frac{1}{6}x+\frac{1}{6}$ is a map in the IFS of $C_1$ and hence $\ell_{\frac{1}{6}x(\sigma(z))+\frac{1}{6}}$ is well defined.
    \item $L_1(z)=\varphi_2(E_{\sigma(z)})$.
    \item $\ell^z$: the line passing through $(0,z)$ and of slope $4$.
    \item For $n\geq 1$, $\varphi_{z_1\cdots z_n}:=\varphi_{z_1}\circ\cdots\circ\varphi_{z_n}$.
\end{itemize}
The last two notations are not needed in the next lemma.

\begin{lemma}\label{lem:ezink}
    For every $z\in C_2$, all of $E_z$, $L_0(z)$ and $L_1(z)$ are line segments in $K$. 
\end{lemma}
Recall that the two endpoints of $\ell_x$ are $(x,0)$ and $(x+\frac{1}{4},1)$, respectively.
\begin{proof}
    By the previous lemma, $E_z, L_0(z), L_1(z)$ are of course subsets of $K$. So it suffices to show that they are all segments. For $3\leq i\leq 5$, $\varphi_i(\ell_{x(z)})$ has an endpoint 
    \[
        \varphi_i\Big( x(z)+\frac{1}{4},1 \Big) = \Big( \frac{x(z)}{6}+\frac{1}{24},\frac{1}{6} \Big)+\Big( \frac{i-3}{24},\frac{i-1}{6} \Big) = \Big( \frac{x(z)}{6}+\frac{i}{24}-\frac{1}{12},\frac{i}{6} \Big)
    \] 
    while $\varphi_{i+1}(\ell_{x(z)})$ has an endpoint
    \[
        \varphi_{i+1}(x(z),0) = \Big( \frac{x(z)}{6},0 \Big)+\Big( \frac{(i+1)-3}{24},\frac{(i+1)-1}{6} \Big) = \Big( \frac{x(z)}{6}+\frac{i}{24}-\frac{1}{12},\frac{i}{6} \Big).
    \] 
    So they share a common point. Due to the `same slope' fact, $E_z=\bigcup_{i=3}^6\varphi_i(\ell_{x(z)})$ is a line segment and hence is $L_{1}(z)$. Similarly, $\varphi_2(\ell_{\frac{1}{6}x(\sigma(z))+\frac{1}{6}})$ has an endpoint 
    \[
        \varphi_2\Big( \frac{x(\sigma(z))}{6}+\frac{1}{6},0 \Big) = \Big( \frac{x(\sigma(z))}{36}+\frac{1}{36},\frac{1}{6} \Big)
    \]
    while $\varphi_1(E_{\sigma(z)})$ has an endpoint 
    \[
        \varphi_1\varphi_6\Big( x(\sigma(z))+\frac{1}{4},1 \Big) = \varphi_1\Big( \frac{x(\sigma(z))}{6}+\frac{1}{6},1 \Big) = \Big( \frac{x(\sigma(z))}{36}+\frac{1}{36},\frac{1}{6} \Big),
    \]
    that is, they share a common point. As a result, $L_0(z)$ is a line segment.
\end{proof}

Moreover, we can conclude from the above lemma that:
\begin{itemize}
    \item $E_z$ has endpoints 
    \begin{equation}\label{eq:ezendpts}
        \varphi_6\Big( x(z)+\frac{1}{4},1 \Big)=\Big( \frac{x(z)}{6}+\frac{1}{6},1 \Big) \text{ and } \varphi_3(x(z),0)= \Big( \frac{x(z)}{6},\frac{1}{3} \Big). 
    \end{equation}
    \item $L_0(z)$ has endpoints $(\frac{x(\sigma(z))}{36}+\frac{5}{72},\frac{1}{3})$ and $(\frac{x(\sigma(z))}{36},\frac{1}{18})$;
    \item $L_1(z)$ has endpoints $(\frac{x(\sigma(z))}{36}+\frac{1}{36},\frac{1}{3})$ and $(\frac{x(\sigma(z))}{36},\frac{2}{9})$;
    \item The above two indicate that for $i\in\{0,1\}$, $L_i(z)$ has endpoints 
    \begin{equation}\label{eq:lzendpts}
        \Big( \frac{x(\sigma(z))}{36}-\frac{i}{24}+\frac{5}{72},\frac{1}{3} \Big) \text{ and } \Big( \frac{x(\sigma(z))}{36},\frac{i}{6}+\frac{1}{18} \Big). 
    \end{equation}
\end{itemize}
Another useful fact is 
\begin{equation}\label{eq:xzsigmaz}
    x(z)-\frac{x(\sigma(z))}{6} = \frac{1}{6}\cdot\Big( -\frac{3}{2}z_1+\frac{5}{2} \Big).
\end{equation}
The lemma below shows that $K$ contains infinitely many line segments hitting the left side of $[0,1]^2$.

\begin{lemma}\label{lem:linesegments}
    $\bigcup_{z\in C_2} (\ell^z\cap[0,1]^2)\subset K$.
\end{lemma}
\begin{proof}
    Fix any $z=\sum_{t=1}^\infty z_t6^{-t} \in C_2$, where $z_t\in\{0,1\}$. We claim that 
    \[
        \ell^z\cap[0,1]^2 = \{(0,z)\}\cup E_{z} \cup \bigcup_{n=1}^\infty \varphi_{z_{1}\cdots z_{n-1}}(L_{z_n}(\sigma^{n-1}(z))),
    \]
    where $\varphi_{z_0}$ denotes the identity map. Combining this with Lemma~\ref{lem:ezink} finishes the proof.
    Since all of the above segments have the same slope, it suffices to show the following three facts.

    Fact 1. $E_z$ and $L_{z_1}(\sigma(z))$ share a common endpoint. In fact, just recall~\eqref{eq:ezendpts}, \eqref{eq:lzendpts} and note that by~\eqref{eq:xzsigmaz}, 
    \begin{align*}
        \frac{x(z)}{6} - \Big( \frac{x(\sigma(z))}{36}-\frac{z_1}{24}+\frac{5}{72} \Big) &= \frac{1}{6}\Big( x(z)-\frac{x(\sigma(z))}{6} \Big) +\frac{z_1}{24}-\frac{5}{72} \\
        &= \frac{1}{36}\cdot\Big( -\frac{3}{2}z_1+\frac{5}{2} \Big)+\frac{z_1}{24}-\frac{5}{72} \\
        &= 0.
    \end{align*}


    Fact 2. For all $n\geq 1$, $\varphi_{z_{1}\cdots z_{n-1}}(L_{z_n}(\sigma^{n-1}(z)))$ and $\varphi_{z_{1}\cdots z_{n}}(L_{z_{n+1}}(\sigma^{n}(z)))$ share a common endpoint. Again recall~\eqref{eq:ezendpts}, \eqref{eq:lzendpts}. Then note that the former segment has an endpoint 
    \begin{equation}\label{eq:tendstoz}
        \begin{aligned}
            \varphi_{z_{1}\cdots z_{n-1}}\Big( \frac{x(\sigma^n(z))}{36}, \frac{z_n}{6}+\frac{1}{18} \Big) = \Big( \frac{x(\sigma^n(z))}{6^{n+1}},\frac{z_n}{6^n}+\frac{1}{3\cdot 6^n} \Big)+\varphi_{z_1\cdots z_{n-1}}(0,0),
        \end{aligned}
    \end{equation}
    while $\varphi_{z_{1}\cdots z_{n}}(T_{z_{n+1}}(\sigma^{n}(z)))$ has an endpoint 
    \begin{align*}
        \varphi_{z_{1}\cdots z_{n}}&\Big( \frac{x(\sigma^{n+1}(z))}{36}-\frac{z_{n+1}}{24}+\frac{5}{72},\frac{1}{3} \Big) \\
        &= \Big( \frac{x(\sigma^{n+1}(z))}{6^{n+2}}-\frac{z_{n+1}}{4\cdot 6^{n+1}}+\frac{5}{2\cdot 6^{n+2}},\frac{1}{3\cdot 6^{n}} \Big)+\varphi_{z_1\cdots z_n}(0,0).
    \end{align*}
    Since 
    \[
        \varphi_{z_1\cdots z_n}(0,0) - \varphi_{z_1\cdots z_{n-1}}(0,0) = \varphi_{z_1\cdots z_{n-1}}\Big(0,\frac{z_n}{6} \Big)- \varphi_{z_1\cdots z_{n-1}}(0,0) = \Big(0,\frac{z_n}{6^n} \Big),
    \]
    those two endpoints have the same $y$-coordinate. Moreover, by~\eqref{eq:xzsigmaz}, the difference between their $x$-coordinates is
    \begin{align*}
        \frac{x(\sigma^n(z))}{6^{n+1}} &- \Big( \frac{x(\sigma^{n+1}(z))}{6^{n+2}}-\frac{z_{n+1}}{4\cdot 6^{n+1}}+\frac{5}{2\cdot 6^{n+2}} \Big) \\
        &= \frac{1}{6^{n+1}}\cdot\Big( x(\sigma^n(z))-\frac{x(\sigma^{n+1}(z))}{6} \Big) + \frac{z_{n+1}}{4\cdot 6^{n+1}}-\frac{5}{2\cdot 6^{n+2}} \\
        &= \frac{1}{6^{n+2}}\cdot\Big( -\frac{3}{2}z_{n+1}+\frac{5}{2} \Big)+ \frac{z_{n+1}}{4\cdot 6^{n+1}}-\frac{5}{2\cdot 6^{n+2}} \\
        &= 0,
    \end{align*}
    that is, they also have the same $x$-coordinate. 

    Fact 3. One of the endpoints of $\varphi_{z_{1}\cdots z_{n-1}}(T_{z_n}(\sigma^{n-1}(z)))$ tends to $(0,z)$ as $n\to\infty$. Since $\lim_{n\to\infty}\varphi_{z_1\cdots z_n}(0,0)=(0,z)$, this follows immediately from~\eqref{eq:tendstoz}.
\end{proof}

\begin{lemma}\label{lem:diffcomponents}
    There exists some connected component of $K$ that is not locally connected.
\end{lemma}
\begin{proof}
    Due to the existence of $\varphi_{19},\ldots,\varphi_{24}$, the right side $\ell_*:=\{1\}\times [0,1]$ is contained in $K$. By Lemma~\ref{lem:linesegments}, $\varphi_1(\ell_*)\cup\varphi_7(\bigcup_{z\in C_2}(\ell^z\cap[0,1]^2))$ belong to one connected component $\mathcal{C}$ of $K$. However, $\varphi_7\varphi_4(\bigcup_{x\in C_1}\ell_x)\subset\varphi_7(\bigcup_{z\in C_2}(\ell^z\cap [0,1]^2))\subset\mathcal{C}$ are infinitely many parallel line segments that are mutually `well separated' (recalling this property of $\{\ell_x:x\in C_1\}$). Moreover, for every $x\in C_1$, there is a sequence $\{x_n\}\subset C_1$ such that the distance between $\ell_{x_n}$ and $\ell_x$ goes to $0$ as $n\to\infty$.
    Therefore, $\mathcal{C}$ is not locally connected at any interior point of $\varphi_7\varphi_4(\ell_x)$.
\end{proof}

\bigskip
\noindent{\bf Acknowledgements.}
This research was conducted during a visit to Professor Huo-Jun Ruan at Zhejiang University. I would like to express my gratitude for his warm hospitality and helpful discussions on the topic of local connectedness.

\small
\bibliographystyle{amsplain}

\end{document}